\documentclass[reqno]{amsart}

\usepackage{amsmath,enumerate}
\usepackage[colorlinks=true,allcolors=black]{hyperref}
\usepackage[nameinlink]{cleveref} 
\usepackage{amsaddr} 

\theoremstyle{plain}
\newtheorem{theorem}{Theorem}[section]
\newtheorem{lemma}[theorem]{Lemma}
\newtheorem{proposition}[theorem]{Proposition}
\newtheorem{corollary}[theorem]{Corollary}
\newtheorem{conjecture}[theorem]{Conjecture}

\theoremstyle{definition}

\newtheorem{example}[theorem]{Example}

\theoremstyle{remark}
\newtheorem{remark}[theorem]{Remark}

\begin{document}

\title[Laplacian spectra of power graphs of certain finite groups]{Laplacian spectra of power graphs of certain finite groups}

\author[Ramesh Prasad Panda]{Ramesh Prasad Panda}

\address{Department of Mathematics\\ Indian Institute of Technology Guwahati\\ Guwahati, Assam - 781039, India\\ rppanda89@gmail.com}

\begin{abstract}
In this article, various aspects of Laplacian spectra of power graphs of finite cyclic, dicyclic and finite $p$-groups are studied. Algebraic connectivity of power graphs of the above groups are considered and determined completely for that of finite $p$-groups. Further, the multiplicity  of Laplacian spectral radius of power graphs of the above groups are studied and determined completely for that of dicyclic and finite $p$-groups. 
The equality of the vertex connectivity and the algebraic connectivity is characterized for power graphs of all of the above groups. Orders of dicyclic groups, for which their power graphs are Laplacian integral, are determined. Moreover, it is proved that the notion of equality of the vertex connectivity and the algebraic connectivity and the notion of Laplacian integral are equivalent for power graphs of dicyclic groups. All possible values of Laplacian eigenvalues  are obtained for power graphs of finite $p$-groups. This shows that power graphs of finite $p$-groups are Laplacian integral. 
\end{abstract}



\thanks{
The author is currently a Post-doctoral fellow at National Institute of Science Education and Research, Bhubaneswar, Odisha - 752050, India.}

\maketitle

\section{Introduction}

In literature, there are various graphs constructed from groups and semigroups, e.g., Cayley graphs (cf. \cite{cayley1878desiderata}), intersection graphs (cf. \cite{zelinka1975intersection}), commuting graphs (cf. \cite{MR687893}) and prime graphs (cf. \cite{MR617092}). The notion of \emph{directed power graph} of a semigroup $S$ was introduced by Kelarev and Quinn \cite{kelarev2000combinatorial} as the directed graph with vertex set $S$ and there is an arc from a vertex $u$ to another vertex $v$ if $v=u^\alpha$ for some $\alpha \in \mathbb{N}$. Followed by this, Chakrabarty et al. \cite{GhoshSensemigroups} defined the {\em power graph} $\mathcal{G}(S)$ of a semigroup $S$ as the (undirected) graph with vertex set $S$ and distinct vertices $u$ and $v$ are adjacent if $v=u^\alpha$ for some $\alpha \in \mathbb{N}$ or $u=v^\beta$ for some $\beta \in \mathbb{N}$. 

In recent years, researchers have studied various properties of power graphs and have shown their usefulness in characterizing finite groups. Cameron and Ghosh \cite{Ghosh} proved that two finite abelian groups with isomorphic power graphs are isomorphic. Cameron \cite{Cameron} proved that if two finite groups have isomorphic power graphs, then they have the same numbers of elements of each order. Curtin and Pourgholi \cite{curtin2014edge} showed that among all finite groups of a given order, the cyclic group of that order has the maximum number of edges. In \cite{hamzeh2017autogroup}, automorphism groups of power graphs of various finite groups were computed. In \cite{ConPowerGr2, ConPowerGr17}, upper bounds of the vertex connectivity along with some cases of equality were found for power graphs of finite cyclic groups. The proper power graph of a group is obtained by deleting the identity element from its power graph. In \cite{bubbo17,doostabadi2015connectivity}, the components of proper power graphs of some finite groups were studied. For more interesting results on power graphs, the reader may refer to \cite{ma2015chromatic, MirzargarPG, MR3200118, MinDegPowerGr17}.

All graphs considered hereafter are undirected, finite and simple (i.e., no loops or parallel edges).

For any graph $\Gamma$ with ordered vertex set $\{v_1, v_2, \ldots, v_n\}$, the \emph{Laplacian matrix} $L(\Gamma)$ of $\Gamma$ is defined as $L(\Gamma) = D(\Gamma) - A(\Gamma)$, where $D(\Gamma)$ is the diagonal matrix whose $(i,i)$th entry is the degree of $v_i$ and $A(\Gamma)$ is the \emph{adjacency matrix} of $\Gamma$ whose $(i,j)$th entry is $1$ if $v_i$ is adjacent to $v_j$ and $0$ otherwise. The matrix $L(\Gamma)$ is symmetric and positive semidefinite, so that its eigenvalues are real and non-negative (cf. \cite{mohar1991laplacian}). Furthermore, the sum of each row (column) of $L(\Gamma)$ is zero, so that it is singular and consequently, its smallest eigenvalue is $0$. The eigenvalues of $L(\Gamma)$ are called the \emph{Laplacian eigenvalues} of $\Gamma$ and are denoted by $\lambda_1(\Gamma) \geq \lambda_2(\Gamma) \geq \ldots \geq \lambda_n(\Gamma)=0$ arranged in non-increasing order. Now let $\lambda_{n_1}(\Gamma) > \lambda_{n_2}(\Gamma) > \ldots > \lambda_{n_r}(\Gamma)=0$ be the distinct Laplacian eigenvalues of $\Gamma$ with multiplicities $m_1, m_2, \ldots, m_r$, respectively. Then the \emph{Laplacian spectrum} of $\Gamma$, that is, the spectrum of $L(\Gamma)$, is represented as
$
\begin{pmatrix}
\lambda_{n_1}(\Gamma) & \lambda_{n_2}(\Gamma) & \cdots &  \lambda_{n_r}(\Gamma) \\
m_1 & m_2 & \cdots & m_r
\end{pmatrix}.
$
 Fiedler \cite{fiedler1973algebraic} termed $\lambda_{n-1}(\Gamma)$ as the \emph{algebraic connectivity} of $\Gamma$, viewing it as a measure of connectivity of $\Gamma$. It is known that  $\lambda_{n-1}(\Gamma) >0$ if and only if $\Gamma$ is connected. The largest Laplacian eigenvalue $\lambda_1(\Gamma)$ is called the \emph{Laplacian spectral radius} of $\Gamma$. A graph is \emph{Laplacian integral} if all its Laplacian eigenvalues are integers.  More results on Laplacian spectra of graphs can be found in the text \cite{cvetkovic2010spectra}.

Recently, researchers have studied various spectral properties of power graphs of finite groups. Chattopadhyay and Panigrahi \cite{chattopadhyay2015laplacian} investigated Laplacian spectra of power graphs of the finite cyclic group $\mathbb{Z}_n$ and the dihedral group $D_n$ (of order $2n$). They showed that the Laplacian spectral radius of the power graph of any finite group $G$ is $|G|$.  Moreover, they expressed the Laplacian characteristic polynomial of $\mathcal{G}(D_n)$ in terms of that of $\mathcal{G}(\mathbb{Z}_n)$. In \cite{sriparna.thesis}, the Laplacian spectrum of power graph of the dicyclic group $Q_n$ (of order $4n$), when $n$ is a power of $2$, was computed. In \cite{power2018radius,Spectra17},  adjacency spectra of power graphs of certain finite groups were studied. Kirkland et al. \cite{MR1873608} supplied an equivalent condition for the equality of vertex connectivity and algebraic connectivity of non-complete and connected graphs in terms of graph join operation. 
   
 In this article, we study Laplacian spectra of power graphs of $\mathbb{Z}_n$, $Q_n$ and  finite $p$-groups.
 For $\mathcal{G}(\mathbb{Z}_n)$, we determine all $n$ for which its algebraic connectivity is $\phi(n)+1$ and all $n$ for which the multiplicity of its Laplacian spectral radius is $\phi(n)+1$. 
  We provide bounds for the algebraic connectivity of $\mathcal{G}(Q_n)$ and find all $n$ for which it is equal to $2$. Moreover, we obtain the multiplicity of the Laplacian spectral radius of $\mathcal{G}(Q_n)$ for all $n$.
   For a finite $p$-group $G$, we show that the statements: (i) the algebraic connectivity of $\mathcal{G}(G)$ is $1$, (ii) the multiplicity of the Laplacian spectral radius of $\mathcal{G}(G)$ is one, and (iii) $G$ is neither cyclic nor generalized quaternion, are all equivalent.
This, when taken together with known results, determines the algebraic connectivity and the multiplicity of the Laplacian spectral radius of $\mathcal{G}(G)$ completely.   
   
     We characterize the equality of the vertex connectivity and the algebraic connectivity for power graphs of $\mathbb{Z}_n$, $Q_n$ and  finite $p$-groups. 
     Furthermore, we determine all $n$ such that $\mathcal{G}(Q_n)$ is Laplacian integral.
      In fact, we prove that the statements: (i) the vertex connectivity and the algebraic connectivity of $\mathcal{G}(Q_n)$ are equal, (ii) the algebraic connectivity of $\mathcal{G}(Q_n)$ is an integer, (iii) $\mathcal{G}(Q_n)$ is Laplacian integral, and (iv) $Q_n$ is generalized quaternion, are all equivalent. When $G$ is a finite $p$-group, we provide iterative methods to find the structure and the Laplacian characteristic polynomial of $\mathcal{G}(G)$. Then we determine all possible values of Laplacian eigenvalues of $\mathcal{G}(G)$, and conclude that power graphs of finite $p$-groups are always Laplacian integral.

\section{Preliminaries}
\label{prelim}

In this section, we state the relevant definitions and recall the necessary results from literature. Additionally, we also fix some notations.

For any graph $\Gamma$, its set of vertices and set of edges are denoted by $V(\Gamma)$ and $E(\Gamma)$, respectively. A graph with just one vertex (hence no edges) is called a \emph{trivial graph}. 
The \emph{complement} $\overline{\Gamma}$ of $\Gamma$ is the graph with vertex set $V(\Gamma)$ and two (distinct) vertices are adjacent if they are non-adjacent in $\Gamma$.
 The \emph{vertex connectivity} $\kappa({\Gamma})$ of $\Gamma$ is the minimum number of vertices whose removal makes $\Gamma$ disconnected or a trivial graph.
  Analogously, the \emph{edge connectivity} $\kappa'({\Gamma})$ of $\Gamma$ is the minimum number of edges whose removal makes $\Gamma$ disconnected or a trivial graph.
   By convention, vertex connectivity and edge connectivity of disconnected graphs or the trivial graph are taken to be $0$.
    The \emph{minimum degree} $\delta(\Gamma)$ of $\Gamma$ is the minimum of degrees of all vertices of $\Gamma$.
     Up to isomorphism, the complete graph on $n$ vertices is denoted by $K_n$. 

In this article, $e$ always denotes the identity element of a (multiplicative) group. Let $G$ be a group and $g \in G$. The order of $g$ in $G$ is denoted by $\mathrm{o}(g)$ and the cyclic subgroup generated by $g$ is denoted by $\langle g \rangle$.
  Let $\approx$ denote the equivalence relation on $G$ defined by $g \approx h$ if $\langle g \rangle = \langle h \rangle$. An equivalence class under $\approx$ is referred to as a $\approx$-class and the $\approx$-class of any $g \in G$ is denoted by $[g]$. For $A \subseteq G$, the subgraph of $\mathcal{G}(G)$ induced by $A$ is denoted by $\mathcal{G}_G(A)$ or simply $\mathcal{G}(A)$. We denote $G^*=G-e$ and $\mathcal{G}^*(G)=\mathcal{G}(G)-e$.
  
   In \cite{GhoshSensemigroups}, Chakrabarty et al. proved that for any group $G$, $\mathcal{G}(G)$ is connected if and only if all elements of $G$ have finite order. Moreover, they determined all finite groups whose power graphs are complete.

\begin{theorem}[{\cite[Theorem 2.12]{GhoshSensemigroups}}]\label{CompleteCond}
	Let $G$ be a finite group. Then $\mathcal{G}(G)$ is complete if and only if $G$ is a cyclic group of order one or $p^\alpha$ for some prime $p$ and $\alpha \in \mathbb{N}$.
\end{theorem}

For $n \in \mathbb{N}$, the additive group of integers modulo $n$ is denoted by $\mathbb{Z}_n = \{\overline{0},\overline{1}, \ldots,$ $\overline{n-1}\}$. We denote the set consisting of the identity element and the generators of $\mathbb{Z}_n$ by $\breve{\mathbb{Z}}_n$, i.e., $\breve{\mathbb{Z}}_n= \{\overline{0} \}  \cup  \left \{\overline{a}\in \mathbb{Z}_n :1 \leq a<n, \gcd(a, n)=1 \right \}$. Further, we denote $\mathbb{Z}'_n=\mathbb{Z}_n-\breve{\mathbb{Z}}_n$ and $\mathcal{G}'(\mathbb{Z}_n)=\mathcal{G}(\mathbb{Z}_n)-\breve{\mathbb{Z}}_n$. Notice that each vertex in $\breve{\mathbb{Z}}_n$ is adjacent to every other vertex of $\mathcal{G}(\mathbb{Z}_n)$.  

We require the following results on power graphs in \Cref{main.section}.

\begin{lemma}[{\cite[Proposition 2.5]{ConPowerGr17}}]\label{SepSetLemma}
	Let $n \geq 2$ be a composite number. Then $\mathcal{G}'(\mathbb{Z}_n)$ is disconnected if and only if $n$ is a product of two distinct primes.
\end{lemma}

\begin{lemma}[{\cite[Theorem 7]{ChattopadhyayConnectivity}}] \label{d.ver.con}
	For any integer $n \geq 2$, $\kappa(\mathcal{G}(Q_n))=2$.
\end{lemma}

\begin{lemma}[{\cite[Corollary 4.2]{MR3200118}}]\label{pgroup.connected}
	For any finite $p$-group $G$, $\mathcal{G}^*(G)$ is connected if and only if $G$ is either cyclic or generalized quaternion.
\end{lemma}

 The isomorphism of graphs and groups is denoted by $\cong$. For $n \in \mathbb{N}$, the number of positive integers that do not exceed $n$ and are relatively prime to $n$ is denoted by $\phi(n)$, and the function $\phi$ is known as \emph{Euler's phi function}. We say that $n$ is a \emph{prime power} if $n=p^\alpha$ for some prime $p$ and $\alpha \in \mathbb{N}$.

For any graph $\Gamma$, the characteristic polynomial $\det(xI - L(\Gamma))$ of $L(\Gamma)$ is called the \emph{Laplacian characteristic polynomial} of $\Gamma$ and is denoted by $\Theta(\Gamma, x)$. If $\Gamma$ is a null graph, for both convenience and consistency, we write $\Theta(\Gamma, x)=1$.

Now we recall some necessary results on Laplacian spectra of graphs.

\begin{theorem}[{\cite[Theorem 7.1.2]{cvetkovic2010spectra}}]\label{Eigen_Compo}
For any graph $\Gamma$, the multiplicity of $0$ as an eigenvalue of $L(\Gamma)$ is equal to the number of components of $\Gamma$.
\end{theorem}

A particular case of \Cref{Eigen_Compo} is the following result due to Fiedler.

\begin{theorem}[{\cite{fiedler1973algebraic}}]\label{algcon.positive}
	A graph $\Gamma$ is connected if and only if $\lambda_{n-1}(\Gamma) > 0$.
\end{theorem}

\begin{theorem}[{\cite[Theorem 3.6]{mohar1991laplacian}}]\label{Eigen_Complemmaent}
For any graph $\Gamma$ on $n$ vertices, $\lambda_n \left( \overline{\Gamma} \right)=0$, and $\lambda_{k}  \left( \overline{\Gamma} \right)=n - \lambda_{n-k}(\Gamma)$ for $1 \leq k \leq n-1$.
\end{theorem}

\begin{theorem}[{\cite{fiedler1973algebraic}}]\label{Eigen_Lambda1}
For any graph $\Gamma$, $\lambda_1(\Gamma) =\displaystyle \max_{1 \leq i \leq r} \lambda_1(\Gamma_i)$, where $\Gamma_1, \ldots, \Gamma_r$ are the components of $\Gamma$.
\end{theorem}

\begin{theorem}[{\cite[Theorem 2.2]{mohar1991laplacian}}]\label{Eigen_Equality}
 If $\Gamma$ is a graph with $n$ vertices, then $\lambda_1(\Gamma) \leq n$. Equality holds if and only if $\overline{\Gamma}$ is not connected.
\end{theorem}

The \emph{union} of graphs $\Gamma_1$ and $\Gamma_2$, denoted by $\Gamma_1 \cup \Gamma_2$, is the graph  with vertex set $V(\Gamma_1) \cup V(\Gamma_2)$ and edge set $E(\Gamma_1) \cup E(\Gamma_2)$. Evidently, union of graphs is associative, so that union of any finite number of graphs can be defined accordingly. If $\Gamma_1$ and $\Gamma_2$ are \emph{disjoint}, i.e., they have no common vertices, we refer to their union as a \emph{disjoint union}, and denote it by $\Gamma_1 + \Gamma_2$. For pairwise disjoint graphs $\Gamma_1, \Gamma_2, \ldots, \Gamma_r$, we denote their union by $\sum_{i=1}^{r} \Gamma_i$. If $\Gamma_1$ and $\Gamma_2$ are disjoint, their \emph{join} $\Gamma_1 \vee \Gamma_2$ is the graph obtained by taking $\Gamma_1 + \Gamma_2$ and adding all edges $\{u,v\}$ with $u \in V(\Gamma_1)$ and $v \in V(\Gamma_2)$. For any graph $\Gamma$, up to isomorphism, $r\Gamma$ denotes the graph obtained by taking disjoint union of $r$ copies of $\Gamma$.

\begin{theorem}[\cite{mohar1991laplacian}]\label{GrLCharProd}
	If $\Gamma$ is the disjoint union of graphs $\Gamma_1,\Gamma_2,\ldots, \Gamma_r$, then 
	\begin{equation*}\label{EqGrLCharProd}
	\Theta(\Gamma,x)= \prod_{i=1}^r \Theta(\Gamma_i,x).
	\end{equation*}
\end{theorem}

\begin{theorem}[\cite{mohar1991laplacian}]\label{GrLCharJoin}
	If $\Gamma_1$ and $\Gamma_2$ are disjoint graphs on $n_1$ and $n_2$ vertices, respectively, then 
	\begin{equation*}\label{EqGrLCharJoin}
	\Theta(\Gamma_1 \vee \Gamma_2,x)= \frac{x(x-n_1-n_2)}{(x-n_1)(x-n_2)} \Theta(\Gamma_1,x-n_2)\Theta(\Gamma_2,x-n_1).
	\end{equation*}
\end{theorem}

In the following results, the Laplacian spectrum was computed for power graphs of cyclic groups of prime power order and generalized quaternion groups. 

\begin{lemma}[{\cite[Corollary 3.3]{chattopadhyay2015laplacian}}]\label{z.prime.power}
	If $n$ is a prime power, then the Laplacian spectrum of $\mathcal{G}(\mathbb{Z}_n)$ is given by
	$\begin{pmatrix}
	0   &  n\\
	1  &  n-1\\
	\end{pmatrix}.$
\end{lemma}

\begin{lemma}[{\cite[Theorem 4.3.3]{sriparna.thesis}}]\label{gen.qua.lap.spec}
	For any integer $\alpha \geq 2$, the Laplacian spectrum of $\mathcal{G}(Q_{2^{\alpha-1}})$ is given by
	$$\begin{pmatrix}
	0   &  2 & 4 & 2^\alpha & 2^{\alpha+1}\\
	1  & 2^{\alpha-1} & 2^{\alpha-1} & 2^\alpha-3 & 2\\
	\end{pmatrix}.$$
\end{lemma}

\section{Laplacian spectra of power graphs}
\label{main.section}

We now present our results on Laplacian spectra of power graphs of finite cyclic, dicyclic and finite $p$-groups in Subsections \ref{cyclic.group}, \ref{dicyclic.group} and \ref{p.group}, respectively. 

By applying \Cref{GrLCharJoin} to $\mathcal{G}(G) \cong K_1 \vee \mathcal{G}^*(G)$, we get  
\begin{equation*}
\Theta(\mathcal{G}(G),x) = \displaystyle \frac{x(x-|G|)}{x-1} \Theta(\mathcal{G}^*(G),x-1).
\end{equation*}

	Thus, if $G$ is a group of order $n \geq 3$, then for any $2 \leq i \leq n-1$,	
	\begin{align}\label{alg.con.bound}
		\lambda_i(\mathcal{G}(G)) = \lambda_{i-1}(\mathcal{G}^*(G)) +1.
	\end{align}
	 
From \eqref{alg.con.bound} and \Cref{algcon.positive}, we have the following lemma.

\begin{lemma}\label{alg.con.con}
	Let $G$ be a finite group of order $n \geq 3$. Then the algebraic connectivity of $\mathcal{G}(G)$ is $1$ if and only if its vertex connectivity is $1$.
\end{lemma}

\vspace{1pt}
\subsection{Finite cyclic group}\label{cyclic.group}~\\
\vspace{-5pt}

In this subsection, we obtain results on the algebraic connectivity and the multiplicity of the Laplacian spectral radius (which is $n$) of $\mathcal{G}(\mathbb{Z}_n)$. Then we characterize the equality of the vertex connectivity and the algebraic connectivity of $\mathcal{G}(\mathbb{Z}_n)$. 

Observe that $\Theta(\mathcal{G}'(\mathbb{Z}_n),x-\phi(n)-1)$ equals with the characteristic polynomial of the submatrix of $L(\mathcal{G}(\mathbb{Z}_n))$ obtained by deleting rows and columns corresponding to the elements of $\breve{\mathbb{Z}}_n$. Thus by \cite[Theorem 2.2]{chattopadhyay2015laplacian}, we have the following useful lemma.

\begin{lemma}\label{z.eigenvalues}
	If the integer $n > 1$ is not a prime number, then
	\begin{equation*}
	\lambda_i(\mathcal{G}(\mathbb{Z}_n)) =
	\begin{cases}
	n & \text{for } 1 \leq i \leq \phi(n)+1, \\
	\lambda_{i-\phi(n)-1}(\mathcal{G}'(\mathbb{Z}_n)) + \phi(n)+1 & \text{for } \phi(n)+2 \leq i \leq n-1, \\
	0	 & \text{for } i=n.
	\end{cases}
	\end{equation*}
\end{lemma}

In \cite{chattopadhyay2015laplacian}, it was shown that the algebraic connectivity of $\mathcal{G}(\mathbb{Z}_n)$ attains its lower bound $\phi(n)+1$ when $n$ is a prime number or product of two distinct primes. In the following theorem, we show that the converse holds as well.

\begin{theorem}\label{alg_con_phi_n}
	For any integer $n>1$, the algebraic connectivity of $\mathcal{G}(\mathbb{Z}_n)$ is $\phi(n)+1$ if and only if $n$ is a prime number or product of two distinct primes.
\end{theorem}

\begin{proof}
	Let the algebraic connectivity of $\mathcal{G}(\mathbb{Z}_n)$ be $\phi(n)+1$. Observe that $\phi(n)+1=n$ if and only if $n$ is a prime number. Now suppose that $n$ is not prime. By \Cref{z.eigenvalues}, the algebraic connectivity of  $\mathcal{G}'(\mathbb{Z}_n)$ is $0$. Thus by \Cref{algcon.positive}, $\mathcal{G}'(\mathbb{Z}_n)$ is disconnected. Hence by applying \Cref{SepSetLemma}, we conclude that $n$ is a product of two distinct primes. Proof of the converse follows from  \cite[Theorem 2.12]{chattopadhyay2015laplacian}.
\end{proof}

It is known from \cite{chattopadhyay2015laplacian} that the multiplicity of $n$ as a Laplacian eigenvalue of $\mathcal{G}(\mathbb{Z}_n)$ is bounded below by $\phi(n)+1$. The next theorem determines all $n \in \mathbb{N}$ for which the equality holds.

\begin{theorem}
For any integer $n>1$, the multiplicity of the Laplacian eigenvalue $n$ of $\mathcal{G}(\mathbb{Z}_n)$ is $\phi(n)+1$ if and only if $n=4$ or $n$ is not a prime power.
\end{theorem}

\begin{proof}
For $n=4$, the multiplicity of $n$ is $n-1 = 3 = \phi(n)+1$ (cf. \Cref{z.prime.power}). Now suppose $n$ is not a prime power. 
Then from \cite[Lemma 2.11]{chattopadhyay2015laplacian}, $\overline{\mathcal{G}'(\mathbb{Z}_n)}$ is connected. Additionally, $|\breve{\mathbb{Z}}_n| = \phi(n)+1$.
 Thus by applying \Cref{Eigen_Equality}, we get $\lambda_1(\mathcal{G}'(\mathbb{Z}_n)) < n-\phi(n)-1$. This together with \Cref{z.eigenvalues} yield $\lambda_i(\mathcal{G}(\mathbb{Z}_n)) < n \text{ for all } \phi(n)+2 \leq i \leq n$. Hence, by following \Cref{z.eigenvalues} once again, the multiplicity of $n$ as a Laplacian eigenvalue of $\mathcal{G}(\mathbb{Z}_n)$ is $\phi(n)+1$.

Conversely, if $n$ is a prime power and $n \neq 4$, then the multiplicity of the Laplacian eigenvalue $n$ of $\mathcal{G}(\mathbb{Z}_n)$ is $n-1$ (cf. \Cref{z.prime.power}), and $n-1 \neq \phi(n)+1$.
\end{proof}

We now determine all $n$ for which the vertex connectivity and the algebraic connectivity of $\mathcal{G}(\mathbb{Z}_n)$ are equal.

\begin{theorem}\label{z.con.algcon}
	For any integer $n>1$, $\kappa(\mathcal{G}(\mathbb{Z}_n)) = \lambda_{n-1}(\mathcal{G}(\mathbb{Z}_n))$ if and only if $n$ is a product of two distinct primes.
\end{theorem}

\begin{proof}
	Let $n$ be a product of two distinct primes. Then from \cite[Theorem 3(ii)]{ChattopadhyayConnectivity} and \cite[Corollary 2.6]{chattopadhyay2015laplacian}, $\kappa(\mathcal{G}(\mathbb{Z}_n)) = \phi(n)+1 = \lambda_{n-1}(\mathcal{G}(\mathbb{Z}_n))$.
		
	Now suppose $n$ is not a product of two distinct primes. If $n$ is a prime power, then considering \Cref{CompleteCond}, $\kappa(\mathcal{G}(\mathbb{Z}_n)) = n-1$ and $ \lambda_{n-1}(\mathcal{G}(\mathbb{Z}_n)) = n$. That is, the equality does not hold.  Thus $n$ has at least two distinct prime factors. We recall from \cite{MR1873608} that for any graph $\Gamma$ on $n$ vertices, if both $\Gamma$ and $\overline{\Gamma}$ are connected, then $\kappa(\Gamma) \neq  \lambda_{n-1}(\Gamma)$. By \Cref{SepSetLemma}, $\mathcal{G}'(\mathbb{Z}_n)$ is connected and by \cite[Lemma 2.11]{chattopadhyay2015laplacian}, $\overline{\mathcal{G}'(\mathbb{Z}_n)}$ is connected. Hence we have $\kappa(\mathcal{G}'(\mathbb{Z}_n)) \neq \lambda_{n-\phi(n)-2}(\mathcal{G}'(\mathbb{Z}_n))$. Additionally, it was ascertained in \cite[Lemma 2.4]{ConPowerGr17} that $\kappa(\mathcal{G}(\mathbb{Z}_n))=\kappa(\mathcal{G}'(\mathbb{Z}_n))+\phi(n)+1$. Consequently, using \Cref{z.eigenvalues}, we get $\kappa(\mathcal{G}(\mathbb{Z}_n)) \neq \lambda_{n-1}(\mathcal{G}(\mathbb{Z}_n))$.
\end{proof}

\vspace{.2cm}
 
\subsection{Dicyclic group}~\\
\label{dicyclic.group}
\vspace{-5pt}

In this subsection, we give bounds of the algebraic connectivity and determine the multiplicity of the Laplacian spectral radius of $\mathcal{G}(Q_n)$. Then we prove that the vertex connectivity and the algebraic connectivity of $\mathcal{G}(Q_n)$ are equal if and only if $\mathcal{G}(Q_n)$ is Laplacian integral. Moreover, we show that the above are equivalent to each of the statements that the algebraic connectivity of $\mathcal{G}(Q_n)$ is $2$ and $Q_n$ is generalized quaternion. We begin by formally stating the definition of a dicyclic group.

For an integer $n \geq 2$, the \emph{dicyclic group} $Q_n$ is a finite group of order $4n$ having presentation
\begin{equation}\label{DicyclicEq}
Q_n=\left \langle a, b \mid a^{2n}=e, a^n=b^2, ab=ba^{-1} \right \rangle,
\end{equation}
where $e$ is the identity element of $Q_n$. When $n$ is a power of $2$, $Q_n$ is called a \emph{generalized quaternion group} of order $4n$. Throughout this subsection, we follow the above presentation of $Q_n$. 

It is known that $(a^ib)^2=a^n$ for all $0 \leq i \leq 2n-1$, and 
\begin{equation}\label{EqQn}
\langle a^ib \rangle=\langle a^{n+i}b \rangle=\{ e, a^ib, a^n, a^{n+i}b \} \text{ for all } 0 \leq i \leq n-1.
\end{equation}

We need the following property of dicyclic groups (cf. \cite{roman2012group}).

\begin{lemma}\label{qn.aib}
	Any element of $Q_n - \langle a \rangle$ is of the form $a^ib$ for some $0 \leq i \leq 2n-1$.
\end{lemma}

\begin{lemma}\label{AlgCon_Dicyclic_LB}
	For any integer $n \geq 2$, the algebraic connectivity of $\mathcal{G}(Q_n)$ satisfies
	\begin{equation*}\label{Eq_AlgCon_Dicyclic_LB}
	1 < \lambda_{4n-1}\left(\mathcal{G}(Q_n)\right)  \leq 2.
	\end{equation*}
\end{lemma}

\begin{proof}
	Considering \Cref{d.ver.con}, $\mathcal{G}^*(Q_n)$ is connected. Thus it follows from \Cref{Eigen_Equality} that $\lambda_1(\overline{\mathcal{G}^*(Q_n)}) < 4n-1$.
	Moreover, by applying \Cref{Eigen_Lambda1}, we have 
	$$\lambda_1(\overline{\mathcal{G}(Q_n)}) = \\ \max\{\lambda_1( \overline{\mathcal{G}^*(Q_n)}), \lambda_1 (\mathcal{G}(\{e\}) ) \} = \lambda_1 ( \overline{\mathcal{G}^*(Q_n)}).$$  
	Accordingly, $\lambda_1(\overline{\mathcal{G}(Q_n)}) < 4n-1$. Hence by \Cref{Eigen_Complemmaent}, $\lambda_{4n-1}(\mathcal{G}(Q_n)) > 1$. In addition to this, by following the proof of \cite[Theorem 4.33]{sriparna.thesis}, $2$ is a Laplacian eigenvalue of  $\mathcal{G}(Q_n)$. We thus get the desired inequality.  
\end{proof}

The following result determines when exactly a dicyclic group is generalized quaternion in terms of its power graph. 

\begin{proposition}\label{d.adj.all}
	For any integer $n\geq 2$, $a^n$ is adjacent to all other vertices of $\mathcal{G}(Q_n)$ if and only if $Q_n$ is generalized quaternion.
\end{proposition}

\begin{proof}
	Let $Q_n$ be generalized quaternion. Then it follows from \Cref{CompleteCond} that $\langle a \rangle \cong \mathbb{Z}_{2n}$ is a clique in $\mathcal{G}(Q_n)$. As a result, $a^n$ is adjacent to all other elements of $\langle a \rangle$ in $\mathcal{G}(Q_n)$.  Moreover, from \eqref{EqQn}, $a^n$ is adjacent to $a^ib$ for all $0 \leq i \leq 2n$. Finally, applying \Cref{qn.aib} we deduce that $a^n$ is adjacent to all other vertices of $\mathcal{G}(Q_n)$.
	
   Whereas, if $Q_n$ is not generalized quaternion, then there exists a prime factor $p>2$ of $n$. Then $a^{\frac{2n}{p}}$ is an element of order $p$. As the order of $a^n$ is $2$, orders of $a^n$ and $a^{\frac{2n}{p}}$ are thus relatively prime. Hence they are not adjacent in $\mathcal{G}(Q_n)$. This concludes proof of the converse.
\end{proof}

Now we obtain the multiplicity of the Laplacian spectral radius of $\mathcal{G}(Q_n)$ for all $n$.

\begin{theorem}\label{Spec_Dic_GQ}
	For any integer $n \geq 2$, the Laplacian eigenvalue $4n$ of $\mathcal{G}(Q_n)$ has multiplicity two if $Q_n$ is generalized quaternion and one otherwise.
\end{theorem}

\begin{proof}
	In view of \eqref{EqQn} and \Cref{qn.aib}, $a^n$ is adjacent to every element of $Q_n - \langle a \rangle$. Moreover, we observe that no element of  $Q_n - \langle a \rangle$ is adjacent to any element of $\langle a \rangle - \{e, a^n \}$ in $\mathcal{G}(Q_n)$. Hence $\overline{\mathcal{G}(Q_n)} - \{e, a^n \}$ is connected. This together with \Cref{d.adj.all} imply that the number of components of $\overline{\mathcal{G}(Q_n)}$ is three if $n$ is a power of $2$ and two otherwise. Accordingly, by \Cref{Eigen_Compo}, the multiplicity of $0$ as a Laplacian eigenvalue of $\overline{\mathcal{G}(Q_n)}$ is three if $n$ is a power of $2$ and two otherwise. Furthermore, by \Cref{Eigen_Complemmaent}, the multiplicity of $4n$ as a Laplacian eigenvalue of $\mathcal{G}(Q_n)$ is equal to one less than the multiplicity of $0$ as a Laplacian eigenvalue of $\overline{\mathcal{G}(Q_n)}$. Hence the result follows.
\end{proof}

We require the following result on the characterization of graphs with equal vertex and algebraic connectivity.

\begin{lemma}[{\cite[Theorem 2.1]{MR1873608}}]\label{kirkland}
	Let $\Gamma$ be a non-complete and connected graph on $n$ vertices. Then $\kappa(\Gamma) = \lambda_{n-1}(\Gamma)$ if and only if $\Gamma$ can be written as $\Gamma_1 \vee \Gamma_2$, where $\Gamma_1$ is a disconnected graph on $n-\kappa(\Gamma)$ vertices and $\Gamma_2$ is a graph on $\kappa(\Gamma)$ vertices with $\lambda_{n-1}(\Gamma_2) \geq 2\kappa(\Gamma)-n$.
\end{lemma}

As mentioned earlier, we next show the equivalence of various properties of Laplacian spectra for power graphs of dicyclic groups.

\begin{theorem}
	For any integer $n\geq 2$, the following statements are equivalent.
	\begin{enumerate}[\rm(i)]
		\item The vertex connectivity and the algebraic connectivity of $\mathcal{G}(Q_n)$ are equal.
		\item The algebraic connectivity of $\mathcal{G}(Q_n)$ is $2$.
		\item The algebraic connectivity of $\mathcal{G}(Q_n)$ is an integer.
		\item $\mathcal{G}(Q_n)$ is Laplacian integral.
		\item $Q_n$ is generalized quaternion.
	\end{enumerate}
\end{theorem}

\begin{proof}
	   	    Suppose $\kappa(\mathcal{G}(Q_n)) = \lambda_{4n-1}(\mathcal{G}(Q_n))$. Then by \Cref{kirkland}, $\mathcal{G}(Q_n)$ can be written as $\Gamma_1 \vee \Gamma_2$ for some graphs $\Gamma_1$ and $\Gamma_2$ with $\Gamma_1$ being isomorphic to a disconnected subgraph of $\mathcal{G}(Q_n)$ on $4n-2$ vertices. We first show that $\{e, a^n\}$ is the unique minimum separating set of $\mathcal{G}(Q_n)$. Let $S$ be a minimum separating set of $\mathcal{G}(Q_n)$.  	Assume that $a^n \notin S$. Note that $e \in S$, and in view of \Cref{d.ver.con}, $|S|=2$. So $S$ contains at most one element of $[a]$. Additionally, as $n \geq 2$, $|[a]|=\phi(2n) \geq 2$. As a result, $\mathcal{G}(\langle a \rangle) -S$ is connected. By following \eqref{EqQn} and \Cref{qn.aib}, all elements of $(\mathcal{G}(Q_n)-\langle a \rangle) -S$ are adjacent to $a^n$. Consequently, $\mathcal{G}(Q_n)-S$ is connected. As this is  a contradiction, $S=\{e, a^n\}$.
	   	    We thus deduce that $\Gamma_1 \cong \mathcal{G}(Q_n) - \{e, a^n\}$ and hence $\mathcal{G}(Q_n)= (\mathcal{G}(Q_n)-\{e, a^n\}) \vee \mathcal{G}(\{e, a^n\})$. As a result, $a^n$ is adjacent to all other vertices of $\mathcal{G}(Q_n)$. Therefore, by applying \Cref{d.adj.all} we conclude that $Q_n$ is generalized quaternion. This proves that (i) implies (v).
	   	    
	   	     If $Q_n$ is generalized quaternion, then it follows from \Cref{gen.qua.lap.spec} that $\mathcal{G}(Q_n)$ is Laplacian integral. Thus (v) implies (iv). Moreover, (iii) follows trivially from (iv).
	   	      If (iii) holds, then by \Cref{AlgCon_Dicyclic_LB}, (ii) holds. Finally, considering \Cref{d.ver.con}, (ii) implies (i).
   	      \end{proof}

\vspace{0.2cm}
\subsection{Finite $p$-group}~\\
\label{p.group}
\vspace{-5pt}

Throughout this subsection, $p$ denotes a prime number.
 A $p$-\emph{group} is a group of order at least two in which order of every element is some power of $p$.
 For a finite $p$-group $G$, we determine the algebraic connectivity and the multiplicity of the Laplacian spectral radius of $\mathcal{G}(G)$.
   Then we characterize the equality of the vertex connectivity and the algebraic connectivity of $\mathcal{G}(G)$.
   Afterwards, we provide iterative methods to obtain the structure and the Laplacian characteristic polynomial of $\mathcal{G}(G)$.
    We find all possible forms of Laplacian eigenvalues of $\mathcal{G}(G)$, thus showing that it is Laplacian integral.

Let $G$ be a group and $g \in G$. We denote $U(g) = \{ h \in G : g \in \langle h \rangle\}$ and $\widehat{U}(g)=U(g)-[g]$. Let $\Gamma(g)$ be the subgraph of $\mathcal{G}(G)$ induced by $U(g)$. Moreover, we denote the component of $\mathcal{G}^*(G)$ containing $g$ by $C(g)$. 
For the above subsets and subgraphs, the underlying group will always be clear from the context.

Since $\Gamma(g)$ is connected, in view of \Cref{Eigen_Compo}, we have the following remark.

\begin{remark}\label{p.Gamma.connected}
	For any group $G$ and $g \in G$, the multiplicity of the Laplacian eigenvalue $0$ of $\Gamma(g)$ is one. 
\end{remark}

For the rest of this subsection, $G$ denotes a finite $p$-group. 

\begin{lemma}\label{p.group.component}
 	If $g$ is an element of order $p$ in  $G$, then $C(g)=\Gamma(g)$.
	\end{lemma}

\begin{proof}
	Since both $\Gamma(g)$ and $C(g)$ are induced subgraphs of $\mathcal{G}(G)$, we only need to show that their vertex sets are equal. Observe that every element of $U(g)$ is a vertex of $C(g)$. To show the reverse inclusion, let $h$ be a vertex of $C(g)$. By \cite[Proposition 3.1]{ConPowerGr17}, $g$ is adjacent to every other vertex of $C(g)$. In particular, $g$ is adjacent to $h$. Since $\mathrm{o}(g)$ is a prime number and $h \neq e$, we have $g \in \langle h \rangle$. Consequently, $h \in U(g)$. This completes the proof of the lemma.
\end{proof}

In the following theorem, when $G$ is neither cyclic nor generalized quaternion, we find the algebraic connectivity and the multiplicity of the Laplacian spectral radius of $\mathcal{G}(G)$. This result, along with \Cref{z.prime.power} and \Cref{gen.qua.lap.spec}, provides algebraic connectivity and the multiplicity of Laplacian spectral radius of power graphs of all finite $p$-groups.     

\begin{theorem}\label{p.cy.gq}
	Let the order of $G$ be $n\geq 3$. Then the following statements are equivalent.
	\begin{enumerate}[\rm(i)]
\item The algebraic connectivity of $\mathcal{G}(G)$ is $1$.	
	\item The multiplicity of the Laplacian eigenvalue $n$ of $\mathcal{G}(G)$ is one.
		\item $G$ is neither cyclic nor generalized quaternion.
	\end{enumerate} 
\end{theorem}

\begin{proof}
	 \Cref{pgroup.connected} and \Cref{alg.con.con} together show that (i) and (iii) are equivalent.
	
	We next show that (ii) and (iii) are equivalent. Let $G$ be neither cyclic nor generalized quaternion. Then by \Cref{pgroup.connected}, $\mathcal{G}^*(G)$ has at least two components.	
	Moreover, from \cite[Proposition 3.2]{ConPowerGr17}, every component of $\mathcal{G}^*(G)$ has exactly $p-1$ vertices of order $p$.
	  As a result, every component of $\mathcal{G}^*(G)$ has at most $n-p$ vertices.  Thus by applying \Cref{Eigen_Lambda1} and \Cref{Eigen_Equality}, we conclude that the Laplacian eigenvalues of $\mathcal{G}^*(G)$ are bounded above by $n-p$. Hence using \eqref{alg.con.bound}, $\lambda_i(\mathcal{G}(G)) \leq n-p+1 < n$ for all $2 \leq i \leq n$. Consequently, the multiplicity of the Laplacian eigenvalue $n$ of $\mathcal{G}(G)$ is one. 
	
	Conversely, let $G$ be either cyclic or generalized quaternion. Then it follows from \Cref{z.prime.power} and \Cref{gen.qua.lap.spec} that the multiplicity of $n$ is at least two. This completes the proof of the theorem.
\end{proof}

Next we show that the vertex connectivity and the algebraic connectivity of power graph of a finite $p$-group are equal exactly when it is not cyclic.

\begin{theorem}
	Let $G$ be of order $n$. Then $\kappa(\mathcal{G}(G)) = \lambda_{n-1}(\mathcal{G}(G))$ if and only if $G$ is not cyclic.
\end{theorem}

\begin{proof}
	Let $G$ be cyclic. Then \Cref{CompleteCond} and \Cref{z.prime.power} together show that the above equality does not hold. 
	
	Now suppose $G$ is not cyclic. If $G$ is generalized quaternion, then by \Cref{gen.qua.lap.spec} and \Cref{d.ver.con}, we get  $\kappa(\mathcal{G}(G))= 2 =\lambda_{n-1}(\mathcal{G}(G))$. Whereas, if $G$ is not generalized quaternion, then it follows from \Cref{pgroup.connected} and \Cref{alg.con.con} that $\kappa(\mathcal{G}(G))= 1 =\lambda_{n-1}(\mathcal{G}(G))$.
\end{proof}

For $g, h \in G$, we say that $[h]$ is a \emph{primitive class} of $g$ if $[g] = [h^p]$ and $h \neq e$. Clearly, if $[h]$ is a primitive class of $g$, then $[h]$ is a primitive class of $g'$ for any $g' \in [g]$. Note that the condition $h \neq e$ is redundant in the above definition when $g \neq e$. We denote the number of primitive classes of any $g \in G$ by $\pi(g)$.

We notice the simple fact that $G$ being a finite $p$-group, for any $g \in G$, we can not have $\pi(g)=0$ and $\mathrm{o}(g)=1$ simultaneously. 

The proof of the following lemma is straightforward. 

\begin{lemma}\label{p.no.primitive}
	If $g \in G$ with $\pi(g) = 0$, then $\Gamma(g)=\mathcal{G}([g]) \cong K_{\phi(\mathrm{o}(g))}$. Consequently, $\Theta(\Gamma(g),x) = x (x-\phi(\mathrm{o}(g)))^{\phi(\mathrm{o}(g))-1}$.
\end{lemma}

The following proposition iteratively describes the structure of the power graph of a finite $p$-group.

\begin{proposition}\label{p.Gamma.structure}	 
	Let $g \in G$, $\pi(g) > 0$ and the distinct primitive classes of $g$ be $[h_1], [h_2], \ldots, [h_{\pi(g)}]$. Then
	\begin{align}\label{eqpGammastructure2}
	\Gamma(g) \cong K_{\phi(\mathrm{o}(g))} \vee \left \{ \Gamma(h_1) + \Gamma(h_2)  + \ldots  + \Gamma(h_{\pi(g)}) \right \}.
	\end{align}
	In particular, for  $g=e$,
	\begin{align}\label{eq.p.Gamma.structure}
	\mathcal{G}(G) \cong K_1 \vee \left \{ \Gamma(h_1) + \Gamma(h_2)  + \ldots  + \Gamma(h_{\pi(e)}) \right \}.
	\end{align}
\end{proposition}

\begin{proof}
	Suppose $\mathrm{o}(g)=p^k$ for some integer $k \geq 0$. Let $h$ be a vertex in $\Gamma(g)$. Then $p^k$ divides $\mathrm{o}(h)$. If $\mathrm{o}(h)=p^k$, then $h \in [g]$. So let $\mathrm{o}(h)=p^l$ for some integer $l > k$. Then $[h^{p^{l-k-1}}]$ is a primitive class of $g$. Thus $[h^{p^{l-k-1}}] = [h_i]$, and hence $h \in U(h_i)$ for some $1 \leq i \leq \pi(g)$. Additionally, as $[g] \subset U(g)$ and $U(h_i) \subset U(g)$ for all $1 \leq i \leq \pi(g)$, we have $U(g) = [g] \cup U(h_1) \cup U(h_2)  \cup \ldots  \cup U(h_{\pi(g)})$. 
	
	Since each vertex in $[g]$ is adjacent to every other vertex of $\Gamma(g)$ and \break $\mathcal{G}([g]) \cong K_{\phi(\mathrm{o}(g))}$, we get $\Gamma(g) \cong K_{\phi(\mathrm{o}(g))} \vee \left \{ \Gamma(h_1) \cup \Gamma(h_2) \cup \ldots  \cup \Gamma(h_{\pi(g)}) \right \}$. For $\pi(g) = 1$, the proof of \eqref{eqpGammastructure2} is thus complete.		
	So let $\pi(g) \geq 2$ and $1 \leq i,j \leq \pi(g)$, $i \neq j$. If possible, suppose $h_i, h_j \in \langle h \rangle$ for some $h \in G$. Since $\langle h \rangle$ is a clique in $\mathcal{G}(G)$, $h_i$ and $h_j$ are adjacent. Thus, as $h_i$ and $h_j$ are of same order, we have $[h_i]=[h_j]$. This is a contradiction. Hence $\Gamma(h_i)$ and $\Gamma(h_j)$ have disjoint vertices. Now, if possible, suppose that $u_i \in U(h_i)$ is adjacent to $u_j \in U(h_j)$ in $\mathcal{G}(G)$. Then without loss of generality, taking $u_i \in \langle u_j \rangle$, we get $u_j \in U(h_i)$. As $\Gamma(h_i)$ and $\Gamma(h_j)$ have disjoint vertices, the above is a contradiction. We therefore have shown \eqref{eqpGammastructure2}. Further, as $\mathcal{G}(G)=\Gamma(e)$, \eqref{eq.p.Gamma.structure} also follows accordingly.	
\end{proof}

Applying \Cref{GrLCharProd} and \Cref{GrLCharJoin} to \Cref{p.Gamma.structure}, we have the following proposition.

\begin{proposition}\label{p.Gamma.charpol}
	 Let $g \in G$, $\pi(g) > 0$ and the distinct primitive classes of $g$ be $[h_1], [h_2], \ldots, [h_{\pi(g)}]$. Then
	\begin{align}\label{eqpGammacharpol12}
	\Theta(\Gamma(g),x) = \frac{x(x-|U(g)|)^{\phi(\mathrm{o}(g))}}{x-\phi(\mathrm{o}(g))} \prod_{i=1}^{\pi(g)} \Theta \left( \Gamma(h_i),x-\phi(\mathrm{o}(g)) \right).
	\end{align}
	In particular, for  $g=e$,
	\begin{equation}\label{eqpGammacharpol}
	\Theta(\mathcal{G}(G),x) = \frac{x(x-|G|)}{x-1} \prod_{i=1}^{\pi(e)} \Theta(\Gamma(h_i),x-1).
	\end{equation}
\end{proposition}

As we shall see, \Cref{p.Gamma.charpol} is important for the study of Laplacian spectra of power graphs of finite $p$-groups. Propositions \ref{p.Gamma.structure} and \ref{p.Gamma.charpol} are illustrated in the following example.

\begin{example}
Consider the group $\mathbb{Z}_{3^2} \times \mathbb{Z}_{3}$, where $\times$ is the direct product of groups.
 For any $(\overline{a}, \overline{b}) \in \mathbb{Z}_{3^2} \times \mathbb{Z}_{3}$, instead of $\Gamma((\overline{a}, \overline{b}))$, we simply write $\Gamma(\overline{a}, \overline{b})$.
 
  The distinct primitive classes of $(\overline{0}, \overline{0})$ are $[(\overline{3}, \overline{0})],[(\overline{3}, \overline{1})],$ $[(\overline{3}, \overline{2})]$ and $[(\overline{0}, \overline{1})]$. Then by \eqref{eq.p.Gamma.structure},
$$\mathcal{G}(\mathbb{Z}_{3^2} \times \mathbb{Z}_{3}) \cong K_1 \vee \{\Gamma(\overline{3}, \overline{0}) + \Gamma(\overline{3}, \overline{1}) + \Gamma(\overline{3}, \overline{2})+\Gamma(\overline{0}, \overline{1})\}.$$
Additionally, $\Gamma(\overline{3}, \overline{1}) \cong \Gamma(\overline{3}, \overline{2}) \cong \Gamma(\overline{0}, \overline{1})\} \cong K_2$. By applying \eqref{eqpGammacharpol}, we thus have
\begin{align*}
\Theta(\mathcal{G}(\mathbb{Z}_{3^2} \times \mathbb{Z}_{3}),x) = x (x-1)^2 (x-3)^3 (x-27) \Theta( \Gamma(\overline{3}, \overline{0}),x-1). 
\end{align*}
Furthermore, as the distinct  primitive classes of $(\overline{3}, \overline{0})$ are $[(\overline{1}, \overline{0})], [(\overline{1}, \overline{1})]$ and $[(\overline{1}, \overline{2})]$, by \eqref{eqpGammastructure2}, we have
$$\Gamma(\overline{3}, \overline{0}) \cong K_2 \vee \{\Gamma(\overline{1}, \overline{0}) + \Gamma(\overline{1}, \overline{1}) + \Gamma(\overline{1}, \overline{2})\}.$$
Notice that $\Gamma(\overline{1}, \overline{0}) \cong \Gamma(\overline{1}, \overline{1}) \cong \Gamma(\overline{1}, \overline{2}) \cong K_6$. Hence application of \eqref{eqpGammacharpol12} gives
\begin{align*}
\Theta(\Gamma(\overline{3}, \overline{0}),x) = x(x-20)^2 (x-2)^2 (x-8)^{15}.
\end{align*}
Therefore, we finally get
$$\mathcal{G}(\mathbb{Z}_{3^2} \times \mathbb{Z}_{3}) \cong K_1 \vee \{ (K_2 \vee 3K_6) + 3K_2\}.$$
and
\begin{align*}
\Theta(\mathcal{G}(\mathbb{Z}_{3^2} \times \mathbb{Z}_{3}),x) = x (x-1)^3 (x-3)^5 (x-9)^{15} (x-21)^2 (x-27). 
\end{align*}
\end{example}

An irreflexive and transitive binary relation $\prec$ on a set $A$ is called \emph{well-founded} if for every non-empty subset $B$ of $A$, there exists $b \in B$ such that there is no $a \in B$ with $a \prec b$. 
	
The following theorem is known as \emph{the principle of well-founded induction}.

\begin{theorem}[{\cite[Theorem 6.10]{jech}}]\label{well.founded}
	Let $\prec$ be a well-founded relation on a set $A$ and $\mathcal{P}$ be a property defined on the elements of $A$. Then $\mathcal{P}$ holds for all elements of $A$ if and only if the following holds: 
	
	given any $a \in A$, if $\mathcal{P}$ holds for all $b \in A$ with $b \prec a$, then $\mathcal{P}$ holds for $a$.
\end{theorem}

The following proposition studies the Laplacian spectrum of $\Gamma(g)$ for any $g \in G$. 

\begin{proposition}\label{p.prop.all.eigenvalue}
	Let $g \in G$ be an element of order $p^{r}$, $r \in \mathbb{N}$, such that $\pi(g)>0$ or $\mathrm{o}(g)>2$. Then every nonzero Laplacian eigenvalue of $\Gamma(g)$ is of the form $\mathrm{o}(g_1)-p^{r-1}$ for some $g_1 \in U(g)$ or $|\widehat{U}(g_2)|+\mathrm{o}(g_2)-p^{r-1}$ for some $g_2 \in U(g)$.
\end{proposition}

\begin{proof}
	Note that the condition $\pi(g)>0$ or $\mathrm{o}(g)>2$ eliminates the case $\pi(g)=0$, $\mathrm{o}(g)=2$. Thus $\Gamma(g)$ has at least one nonzero Laplacian eigenvalue.

	Let $p^k$ be the least common multiple of orders of all elements of $G$. As $G$ is a finite $p$-group, this means that it has an element of order $p^k$.
	 Observe that the usual relation $<$ is well-founded on $\{ l \in \mathbb{Z} : 0 \leq l \leq k-1 \}$. In order to prove the proposition, we apply the principle of well-founded induction on the above set and prove the following: for any element $g$ of order $p^{k-l}$, every nonzero Laplacian eigenvalue of $\Gamma(g)$ is of the form $\mathrm{o}(g_1)-p^{k-l-1}$ for some $g_1 \in U(g)$ or $|\widehat{U}(g_2)|+\mathrm{o}(g_2)-p^{k-l-1}$ for some $g_2 \in U(g)$.

	If $g$ is an element of order $p^k$, then $\pi(g)=0$. Thus $\mathrm{o}(g)-p^{k-1}=|\widehat{U}(g)|+\mathrm{o}(g)-p^{k-1}=\phi(p^k)$, and by \Cref{p.no.primitive}, this is the only Laplacian eigenvalue of $\Gamma(g)$. This proofs our claim for $l=0$. In fact, if $k=1$, then the proof is complete. So let $k > 1$. We now assume that the assertion holds for $l=m$, $0 \leq m \leq k-2$, 
	and show it for $l=m+1$. 
	
	Let $g$ be an element of order $p^{k-(m+1)}$. If $\pi(g)=0$, then the statement holds for $l=m+1$ by an argument similar to the case when $g$ is of order $p^k$. Now let $\pi(g)>0$ and the distinct primitive classes of $g$ be $[h_1], [h_2], \ldots, [h_{\pi(g)}]$. Observe that for every $1 \leq i \leq \pi(g)$, $\mathrm{o}(g)-p^{k-m-2} = \phi(p^{k-m-1})$ is a root of $\Theta \left( \Gamma(h_i),x-\phi(p^{k-m-1}) \right)$. Additionally, as $g \neq e$, $\mathrm{o}(h_i) > 2$ for all $1 \leq i \leq \pi(g)$. Thus in view of the induction hypothesis, every root of $\Theta \left( \Gamma(h_i),x-\phi(p^{k-m-1}) \right)$, other than $\mathrm{o}(g)-p^{k-m-2}$, is of the form $\mathrm{o}(g_{1})-p^{k-m-2}$ for some $g_{1} \in U(h_i)$ or $|\widehat{U}(g_{2})|+\mathrm{o}(g_{2})-p^{k-m-2}$ for some $g_{2} \in U(h_i)$. We also note that $|U(g)|=|\widehat{U}(g)|+\mathrm{o}(g)-p^{k-m-2}$. Hence, application of \eqref{eqpGammacharpol12} proves the assertion for $l=m+1$. This completes the proof of the proposition.
\end{proof}

Now we provide all possible forms of Laplacian eigenvalues of $\mathcal{G}(G)$ and conclude that it is Laplacian integral.

\begin{theorem}\label{p.eigenvalues} 
	Every Laplacian eigenvalue of $\mathcal{G}(G)$ is among $0$, $\mathrm{o}(g)$ for some $g \in G$ or $|\widehat{U}(h)|+\mathrm{o}(h)$ for some $h \in G$. In particular,  $\mathcal{G}(G)$ is Laplacian integral.
\end{theorem}

\begin{proof}
	 Trivially, $|\widehat{U}(e)|+\mathrm{o}(e)= n$. Let the distinct primitive classes of $e$ be $[g_1], [g_2], \ldots, [g_{\pi(e)}]$.
	  Then by \Cref{p.prop.all.eigenvalue}, if $\pi(g_i)>0$ or $\mathrm{o}(g_i)>2$ for any $1 \leq i \leq \pi(e)$, then every nonzero root of $\Theta \left( \Gamma(g_i),x-1 \right)$ is of the form $\mathrm{o}(h_1)$ for some $h_1 \in U(g_i)$ or $|\widehat{U}(h_2)|+\mathrm{o}(h_2)$ for some $h_2 \in U(g_i)$. Moreover, for every $1 \leq i \leq \pi(e)$, $1$ is a root of $\Theta \left( \Gamma(g_i),x-1 \right)$.
	   Hence by applying \eqref{eqpGammacharpol}, the proof follows.
\end{proof}

\begin{remark}
	For any $g \in G$, $|\widehat{U}(g)|+\mathrm{o}(g)=\mathrm{o}(g)$ if and only if $\pi(g) = 0$. 
\end{remark}

In the following, we give some properties of the Laplacian eigenvalue \break $|\widehat{U}(g)|+\mathrm{o}(g)$.

\begin{proposition}\label{p.eigenvalue.U.p.p}
	For any $g \in G$, the following statements hold. 
	\begin{enumerate}[\rm(i)]
		\item $|\widehat{U}(g)|+\mathrm{o}(g)$ is a multiple of $\mathrm{o}(g)$.
		\item If $|\widehat{U}(g)|+\mathrm{o}(g)$ is a prime power, then $\pi(g)=0$ or $\pi(g) \equiv 1\pmod p$.
	\end{enumerate}
\end{proposition}

\begin{proof}
	Let $\mathrm{o}(g)=p^k$ for some $k \in \mathbb{N}$. 
	
	\noindent
	(i) If $\pi(g)=0$, then $|\widehat{U}(g)|+\mathrm{o}(g) =p^k$. Now let $\pi(g) > 0$. Then in view of \Cref{p.Gamma.structure}, there exist positive integers $a_1, \ldots, a_l$ with $a_1=\pi(g)$ such that	
	\begin{align}\label{p.eigenvalue.U.p2}
	& |U(g)| =\phi(p^k)+ a_1\phi(p^{k+1})+\ldots+a_l\phi(p^{k+l}) \nonumber\\
	\Rightarrow & |\widehat{U}(g)|+\mathrm{o}(g) = p^k+ a_1\phi(p^{k+1})+\ldots+a_l\phi(p^{k+l}).
	\end{align}
	Thus we have shown (i).
	
	\smallskip 
	
	\noindent
	(ii) Let $|\widehat{U}(g)|+\mathrm{o}(g)=p^m$ for some $m \in \mathbb{N}$. 
	If $m=k$, then $\pi(g)=0$. Now let $m > k$. Then we have $\pi(g) > 0$. Comparing both sides of \eqref{p.eigenvalue.U.p2}, we thus get $p | (a_1-1)$. This proves (ii).
\end{proof}

Using \Cref{p.eigenvalues} and \Cref{p.eigenvalue.U.p.p}(i), we have the following corollary.

\begin{corollary}\label{p.all.eigenvalues}
	
	Any nonzero Laplacian eigenvalue of $\mathcal{G}(G)$ is $1$ or multiple of the order of a non-identity element of $G$. In particular, it is $1$ or multiple of a positive power of $p$.
\end{corollary}

We conclude this subsection with the following proposition, which also serves as an illustration for some of the above results. 

\begin{proposition}
	If $G$ is a group of order $p^2$, then the Laplacian spectrum of $\mathcal{G}(G)$ is either
	$$\begin{pmatrix}
	0   &  p^2\\
	1  & p^2-1\\
	\end{pmatrix} \text{ or }
	\begin{pmatrix}
	0   &  1 & p & p^2\\
	1  & p & (p+1)(p-2) & 1\\
	\end{pmatrix}.$$
\end{proposition}

\begin{proof}
	As already recalled in \Cref{z.prime.power}, if $G$ is cyclic, then the Laplacian spectrum of $\mathcal{G}(G)$ is 
	$\begin{pmatrix}
	0   &  p^2\\
	1  &  p^2-1\\
	\end{pmatrix}$. It is known that any group of order $p^2$ is abelian (cf. \cite{roman2012group}). Thus, if $G$ is not cyclic, then $G \cong \mathbb{Z}_p \times  \mathbb{Z}_p$. In $\mathbb{Z}_p \times  \mathbb{Z}_p$, the distinct primitive classes of $(\overline{0},\overline{0})$ are $[(\overline{0}, \overline{1})],[( \overline{1}, \overline{0} )]$, $[( \overline{1}, \overline{1} ) ],[( \overline{1}, \overline{2} ) ], \ldots, [( \overline{1}, \overline{p-1} ) ]$. Moreover, we observe that each of these classes induce subgraphs isomorphic to $K_{p-1}$ in $\mathcal{G}(\mathbb{Z}_p \times  \mathbb{Z}_p)$. Hence by applying \eqref{eqpGammacharpol}, we get
	\begin{align*}
	\Theta(\mathcal{G}(G),x) & = \frac{x(x-p^2)}{x-1} \{\Theta(K_{p-1},x-1)\}^{p+1}\\
	& = x(x-1)^p (x-p)^{(p+1)(p-2)} (x-p^2).
	\end{align*}
	This proves the proposition.  
\end{proof}

\section{Conclusion}

In this article, we obtained characterization of the equality of vertex and algebraic connectivity of power graphs of $\mathbb{Z}_n$, $Q_n$ and finite $p$-groups. Providing a group theoretic characterization of the above equality for all finite groups is still open for  study. Moreover, we proved that the power graph of $Q_n$ is Laplacian integral if and only if $Q_n$ is generalized quaternion, and that the power graph of a finite  $p$-group is always Laplacian integral. Based on our observations, we state the following for $\mathbb{Z}_n$.

\begin{conjecture}
		For any integer $n\geq 2$, the following  statements are equivalent.
		\begin{enumerate}[\rm(i)]
			\item The algebraic connectivity of $\mathcal{G}(\mathbb{Z}_n)$ is an integer.
			\item $\mathcal{G}(\mathbb{Z}_n)$ is Laplacian integral.
			\item $n$ is a prime power or product of two primes.
		\end{enumerate}
\end{conjecture}

\section{Acknowledgment}
The author is thankful to Dr. K. V. Krishna and Prof. A. R. Moghaddamfar for their valuable comments which have improved the article. 	During the preparation of the article, the author was supported by the institute fellowship of IIT Guwahati, India.
 

\begin{thebibliography}{10}
	
	\bibitem{MR687893}
	E.~A. Bertram.
	\newblock Some applications of graph theory to finite groups.
	\newblock {\em Discrete Math.}, 44(1):31--43, 1983.
	
	\bibitem{bubbo17}
	D.~Bubboloni, M.~A. Iranmanesh, and S.~M. Shaker.
	\newblock On some graphs associated with the finite alternating groups.
	\newblock {\em Comm. Algebra}, 45(12):5355--5373, 2017.
	
	\bibitem{Cameron}
	P.~J. Cameron.
	\newblock The power graph of a finite group, {II}.
	\newblock {\em J. Group Theory}, 13(6):779--783, 2010.
	
	\bibitem{Ghosh}
	P.~J. Cameron and S.~Ghosh.
	\newblock The power graph of a finite group.
	\newblock {\em Discrete Math.}, 311(13):1220--1222, 2011.
	
	\bibitem{cayley1878desiderata}
	P.~Cayley.
	\newblock Desiderata and {S}uggestions: {N}o. 2. {T}he {T}heory of {G}roups:
	{G}raphical {R}epresentation.
	\newblock {\em Amer. J. Math.}, 1(2):174--176, 1878.
	
	\bibitem{GhoshSensemigroups}
	I.~Chakrabarty, S.~Ghosh, and M.~K. Sen.
	\newblock Undirected power graphs of semigroups.
	\newblock {\em Semigroup Forum}, 78(3):410--426, 2009.
	
	\bibitem{sriparna.thesis}
	S.~Chattopadhyay.
	\newblock Some graph theoretic and spectral results on power graphs of certain
	finite groups, {PhD Thesis}.
	\newblock IIT Kharagpur, India, 2015.
	
	\bibitem{ChattopadhyayConnectivity}
	S.~Chattopadhyay and P.~Panigrahi.
	\newblock Connectivity and planarity of power graphs of finite cyclic, dihedral
	and dicyclic groups.
	\newblock {\em Algebra Discrete Math.}, 18(1):42--49, 2014.
	
	\bibitem{chattopadhyay2015laplacian}
	S.~Chattopadhyay and P.~Panigrahi.
	\newblock On {L}aplacian spectrum of power graphs of finite cyclic and dihedral
	groups.
	\newblock {\em Linear Multilinear Algebra}, 63(7):1345--1355, 2015.
	
	\bibitem{power2018radius}
	S.~Chattopadhyay, P.~Panigrahi, and F.~Atik.
	\newblock Spectral radius of power graphs on certain finite groups.
	\newblock {\em Indag. Math. (N.S.)}, 29(2):730--737, 2018.
	
	\bibitem{ConPowerGr2}
	S.~Chattopadhyay, K.~L. Patra, and B.~K. Sahoo.
	\newblock Vertex connectivity of the power graph of a finite cyclic group.
	\newblock {\em Discrete Appl. Math.}, 2018.
	\newblock DOI:10.1016/j.dam.2018.06.001.
	
	\bibitem{curtin2014edge}
	B.~Curtin and G.~R. Pourgholi.
	\newblock Edge-maximality of power graphs of finite cyclic groups.
	\newblock {\em J. Algebraic Combin.}, 40(2):313--330, 2014.
	
	\bibitem{cvetkovic2010spectra}
	D.~Cvetkovi\'c, P.~Rowlinson, and S.~Simi\'c.
	\newblock {\em An introduction to the theory of graph spectra}, volume~75 of
	{\em London Mathematical Society Student Texts}.
	\newblock Cambridge University Press, Cambridge, 2010.
	
	\bibitem{doostabadi2015connectivity}
	A.~Doostabadi and M.~F.~D. Ghouchan.
	\newblock On the connectivity of proper power graphs of finite groups.
	\newblock {\em Comm. Algebra}, 43(10):4305--4319, 2015.
	
	\bibitem{fiedler1973algebraic}
	M.~Fiedler.
	\newblock Algebraic connectivity of graphs.
	\newblock {\em Czechoslovak Math. J.}, 23(98):298--305, 1973.
	
	\bibitem{hamzeh2017autogroup}
	A.~Hamzeh and A.~R. Ashrafi.
	\newblock Automorphism groups of supergraphs of the power graph of a finite
	group.
	\newblock {\em European J. Combin.}, 60:82--88, 2017.
	
	\bibitem{jech}
	T.~Jech.
	\newblock {\em Set theory}.
	\newblock Springer Monographs in Mathematics. Springer-Verlag, 2003.
	
	\bibitem{kelarev2000combinatorial}
	A.~V. Kelarev and S.~J. Quinn.
	\newblock A combinatorial property and power graphs of groups.
	\newblock In {\em Contributions to general algebra, 12 ({V}ienna, 1999)}, pages
	229--235. Heyn, Klagenfurt, 2000.
	
	\bibitem{MR1873608}
	S.~J. Kirkland, J.~J. Molitierno, M.~Neumann, and B.~L. Shader.
	\newblock On graphs with equal algebraic and vertex connectivity.
	\newblock {\em Linear Algebra Appl.}, 341:45--56, 2002.
	
	\bibitem{ma2015chromatic}
	X.~Ma and M.~Feng.
	\newblock On the chromatic number of the power graph of a finite group.
	\newblock {\em Indag. Math. (N.S.)}, 26(4):626--633, 2015.
	
	\bibitem{Spectra17}
	Z.~Mehranian, A.~Gholami, and A.~R. Ashrafi.
	\newblock The spectra of power graphs of certain finite groups.
	\newblock {\em Linear Multilinear Algebra}, 65(5):1003--1010, 2017.
	
	\bibitem{MirzargarPG}
	M.~Mirzargar, A.~R. Ashrafi, and M.~J. Nadjafi-Arani.
	\newblock On the power graph of a finite group.
	\newblock {\em Filomat}, 26(6):1201--1208, 2012.
	
	\bibitem{MR3200118}
	A.~R. Moghaddamfar, S.~Rahbariyan, and W.~J. Shi.
	\newblock Certain properties of the power graph associated with a finite group.
	\newblock {\em J. Algebra Appl.}, 13(7):1450040, 18, 2014.
	
	\bibitem{mohar1991laplacian}
	B.~Mohar.
	\newblock The {L}aplacian spectrum of graphs.
	\newblock In {\em Graph theory, combinatorics, and applications. {V}ol.\ 2
		({K}alamazoo, {MI}, 1988)}, Wiley-Intersci. Publ., pages 871--898. Wiley, New
	York, 1991.
	
	\bibitem{ConPowerGr17}
	R.~P. Panda and K.~V. Krishna.
	\newblock On connectedness of power graphs of finite groups.
	\newblock {\em J. Algebra Appl.}, 17(10):1850184, 20, 2018.
	
	\bibitem{MinDegPowerGr17}
	R.~P. Panda and K.~V. Krishna.
	\newblock On minimum degree, edge-connectivity and connectivity of power graphs
	of finite groups.
	\newblock {\em Comm. Algebra}, 46(7):3182--3197, 2018.
	
	\bibitem{roman2012group}
	S.~Roman.
	\newblock {\em Fundamentals of group theory}.
	\newblock Birkh\"auser/Springer, New York, 2012.
	
	\bibitem{MR617092}
	J.~S. Williams.
	\newblock Prime graph components of finite groups.
	\newblock {\em J. Algebra}, 69(2):487--513, 1981.
	
	\bibitem{zelinka1975intersection}
	B.~Zelinka.
	\newblock Intersection graphs of finite abelian groups.
	\newblock {\em Czechoslovak Math. J.}, 25(100):171--174, 1975.
	
\end{thebibliography}

\end{document}